\documentclass[11pt]{amsart}
\usepackage{amsmath, amssymb, amsthm, amsfonts}
\usepackage{hyperref,doi}
\usepackage{epsf}
\usepackage{enumerate}
\usepackage{graphicx}
\usepackage{array}
\usepackage[margin=1in]{geometry}

\usepackage{color}

\newcommand{\ZZ}{\ensuremath{\mathbb{Z}}}

\DeclareMathOperator{\stab}{stab}

\newcommand{\R}{\mathbb{R}}

\newcommand{\rstablenk}{\Delta_{n,k}^{\stab(r)}}
\newcommand{\pnk}{P_{n,k}^{\stab(r)}}
\newcommand{\trinkr}{\nabla_{n,k}^r}
\newcommand{\rstablenkminus}{\Delta_{n,k}^{\stab(r-1)}}
\newcommand{\trinkrminus}{\nabla_{n,k}^{r-1}}
\newcommand{\trinkrdifference}{\max\nabla_{n,k}^{r-1}\backslash\max\nabla_{n,k}^{r}}

\newcommand{\floorminusstablenk}{P_{n,k}^{\stab\left(\left\lfloor\frac{n}{k}\right\rfloor-1\right)}}

\newcommand{\twostablenk}{\Delta_{n,k}^{\stab(2)}}

\newcommand{\hypernk}{\Delta_{n,k}}
\newcommand{\trink}{\nabla_{n,k}}

\newcommand{\hl}{H_\ell}

\newcommand{\hlhalf}{H_{\ell}^{(+)}}
\newcommand{\phlhalf}{K_{\ell}^{(+)}}

\newcommand{\phlnhalf}{K_n^{(-)}}
\newcommand{\hlr}{H_{\ell,r}}

\newcommand{\hlrhalf}{H_{\ell,r}^{(-)}}
\newcommand{\phlrhalf}{K_{\ell,r}^{(-)}}
\newcommand{\phlrhalfn}{\widetilde{K}_{\ell,r}^{(+)}}
\newcommand{\hlrminus}{H_{\ell,r-1}}
\newcommand{\nkfloor}{\left\lfloor\frac{n}{k}\right\rfloor}
\newcommand{\nkceiling}{\left\lceil\frac{n}{k}\right\rceil}

\newtheorem{thm}{Theorem}[section]
 \newtheorem{cor}[thm]{Corollary}
 \newtheorem{lem}[thm]{Lemma}
 
 \theoremstyle{definition}
 \newtheorem{defn}[thm]{Definition}
 \theoremstyle{remark}
 \newtheorem{rem}[thm]{Remark}
 
 \numberwithin{equation}{section}

\DeclareMathOperator{\relint}{relint}

\begin{document}

\begin{abstract}
Let $k, n$ and $r$ be positive integers with $k < n$ and $r\leq\nkfloor$. 
We determine the facets of the $r$-stable $n,k$-hypersimplex.  
As a result, it turns out that the $r$-stable $n,k$-hypersimplex has exactly $2n$ facets for every $r<\nkfloor$.  
We then utilize the equations of the facets to study when the $r$-stable hypersimplex is Gorenstein.
For every $k>0$ we identify an infinite collection of Gorenstein $r$-stable hypersimplices, consequently expanding the collection of $r$-stable hypersimplices known to have unimodal Ehrhart $\delta$-vectors.
\end{abstract}

\title{Facets of the $r$-stable $n,k$-hypersimplex}
\author{Takayuki Hibi}
\address{Department of Pure and Applied Mathematics\\
Graduate School of Information Science and Technology\\
Osaka University, Toyonaka, Osaka 560--0043, Japan}
\email{hibi@math.sci.osaka-u.ac.jp}

\author{Liam Solus}
\address{Department of Mathematics\\
         University of Kentucky\\
         Lexington, KY 40506--0027, USA}
\email{liam.solus@uky.edu}

\date{1 June 2015}

\thanks{Liam Solus was supported by a 2014 National Science Foundation/Japan Society for the Promotion of Science East Asia and Pacific Summer Institute Fellowship.}




\maketitle

\section{Introduction}

The $(n,k)$-hypersimplices are an important collection of integer polytopes arising naturally in the settings of convex optimization, matroid theory, combinatorics, and algebraic geometry.  
Generalizing the standard $(n-1)$-simplex, the $(n,k)$-hypersimplices serve as a useful collection of examples in these various contexts.  
While these polytopes are well-studied, there remain interesting open questions about their properties in the field of Ehrhart theory, the study of integer point enumeration in dilations of rational polytopes  (see for example \cite{del}).  
The $r$-stable $(n,k)$-hypersimplices are a collection of lattice polytopes within the $(n,k)$-hypersimplex that were introduced in \cite{braunsol} for the purpose of studying unimodality of the Ehrhart $\delta$-polynomials of the $(n,k)$-hypersimplices.  
However, they also exhibit interesting geometric similarities to the $(n,k)$-hypersimplices which they generalize.  
For example, it is shown in \cite{braunsol} that a regular unimodular triangulation of $(n,k)$-hypersimplex, called the circuit triangulation, restricts to a triangulation of each $r$-stable $(n,k)$-hypersimplex.  

In the present paper, we compute the facets of the $r$-stable $(n,k)$-hypersimplices for $1\leq r<\nkfloor$ and then study when they are Gorenstein. 
In section \ref{facets}, we compute their facet-defining inequalities (Theorem \ref{main theorem}).  
From these computations, we see that the geometric similarities between the $(n,k)$-hypersimplex and the $r$-stable $(n,k)$-hypersimplices within are apparent in their minimal $H$-representations.  
Moreover, it turns out that each $r$-stable $(n,k)$-hypersimplex has exactly $2n$ facets for $1\leq r<\nkfloor$ (Corollary \ref{2nfacets}). 
In section \ref{gorenstein}, we classify $1 \leq r < \left\lfloor \frac{n}{k} \right\rfloor$ for which these polytopes are Gorenstein (Theorem \ref{gorensteinrstablethm}).
We conclude that the Ehrhart $\delta$-vector of each Gorenstein $r$-stable hypersimplex is unimodal (Corollary \ref{unimodalgorenstein}), thereby expanding the collection or $r$-stable hypersimplices known to have unimodal $\delta$-polynomials.

\section{The $H$-representation of the $r$-stable $(n,k)$-hypersimplex}
\label{facets}

We first recall the definitions of the $(n,k)$-hypersimplices and the $r$-stable $(n,k)$-hypersimplices.  
For integers $0<k<n$ let $[n]:=\{1,2,\ldots,n\}$ and let ${[n] \choose k}$ denote the collection of all $k$-subsets of $[n]$. 
The \emph{characteristic vector} of a subset $I$ of $[n]$ is 
the $(0,1)$-vector 
$\epsilon_I=(\epsilon_1,\ldots,\epsilon_n)$ 
for which $\epsilon_i=1$ for $i\in I$ and $\epsilon_i=0$ for $i\notin I$.   
The \emph{$(n,k)$-hypersimplex} is the convex hull in $\R^n$ of the collection of characteristic vectors $\{\epsilon_I : I\in{[n]\choose k}\}$, and it is denoted $\hypernk$.  
Label the vertices of a regular $n$-gon embedded in $\R^2$ 
in a clockwise fashion from $1$ to $n$.  
Given a third integer $1\leq r\leq\nkfloor$, a subset $I\subset[n]$ (and its characteristic vector) is called \emph{$r$-stable} if, for each pair $i,j\in I$, 
the path of shortest length from $i$ to $j$ about the $n$-gon uses 
at least $r$ edges.  
The \emph{$r$-stable $n,k$-hypersimplex}, denoted by $\rstablenk$, 
is the convex polytope in $\R^{n}$ which is the convex hull
of the characteristic vectors of all $r$-stable $k$-subsets of $[n]$. 
For a fixed $n$ and $k$ the $r$-stable $(n,k)$-hypersimplices form the nested chain of polytopes
\begin{equation*}
\Delta_{n,k}\supset\Delta_{n,k}^{\stab(2)}\supset\Delta_{n,k}^{\stab(3)}\supset\cdots\supset\Delta_{n,k}^{\stab\left(\left\lfloor\frac{n}{k}\right\rfloor\right)}.
\end{equation*}
Notice that $\hypernk$ is precisely the $1$-stable $(n,k)$-hypersimplex.

The definitions of $\hypernk$ and $\rstablenk$ provided are $V$-representations of these polytopes.  
In this section we provide the minimal $H$-representation of $\rstablenk$, i.e. its collection of facet-defining inequalities.  
It is well-known that the facet-defining inequalities of $\hypernk$ are $\sum_{i=1}^nx_i=k$ together with $x_\ell\geq0$ and $x_\ell\leq1$ for all $\ell\in[n]$.  
Let $H$ denote the hyperplane in $\R^n$ defined by the equation
$\sum_{i=1}^nx_i=k.$
For $1\leq r\leq \nkfloor$  and $\ell\in[n]$ consider the closed convex subsets of $\R^n$ 
$$
\hlhalf:=\left\{(x_1,x_2,\ldots,x_n)\in\R^n: x_\ell\geq0\right\}\cap H, \mbox{ and}
$$
$$
\hlrhalf:=\left\{(x_1,x_2,\ldots,x_n)\in\R^n: \sum_{i=\ell}^{\ell+r-1}x_i\leq1\right\}\cap H.
$$
In the definition of $\hlrhalf$ the indices $i$ of the coordinates $x_1,\ldots,x_n$ are taken to be elements of $\ZZ/n\ZZ$.  
We also let $\hl$ and $\hlr$ denote the $(n-2)$-flats given by strict equality in the above definitions.  
In the following we will say an $(n-2)$-flat is \emph{facet-defining} (or \emph{facet-supporting}) for $\rstablenk$ if it contains a facet of $\rstablenk$.

\begin{thm}
\label{main theorem}
Let $1<k<n-1$.  
For $1\leq r<\nkfloor$ the facet-defining inequalities for $\rstablenk$ are $\sum_{i=1}^nx_i=k$ together with $\sum_{i=\ell}^{\ell+r-1}x_i\leq1$ and $x_\ell\geq0$ for $\ell\in[n]$.
In particular,
$$
\rstablenk = \bigcap_{\ell\in[n]}\hlhalf\cap\bigcap_{\ell\in[n]}\hlrhalf.
$$
\end{thm}

The following is an immediate corollary to these results.
\begin{cor}\label{2nfacets}
All but possibly the smallest polytope in the nested chain
\begin{equation*}
\Delta_{n,k}\supset\Delta_{n,k}^{\stab(2)}\supset\Delta_{n,k}^{\stab(3)}\supset\cdots\supset\Delta_{n,k}^{\stab\left(\left\lfloor\frac{n}{k}\right\rfloor\right)}
\end{equation*}
has 2n facets.  
\end{cor}
This is an interesting geometric property since the number of vertices of these polytopes strictly decreases down the chain.
To prove Theorem~\ref{main theorem} we will utilize the geometry of the circuit triangulation of $\hypernk$ as defined in \cite{lam}, the construction of which we will now recall.


\subsection{The circuit triangulation.}
Fix $0<k<n$, and let $G_{n,k}$ be the labeled, directed graph with the following vertices and edges.  
The vertices of $G_{n,k}$ are all the vectors $\epsilon_I\in\R^n$ where $I$ is a $k$-subset of $[n]$.  
We think of the indices of a vertex of $G_{n,k}$ modulo $n$.
Now suppose that $\epsilon$ and $\epsilon^\prime$ are two vertices of $G_{n,k}$ such that for some $i\in[n]$ $(\epsilon_i,\epsilon_{i+1})=(1,0)$ and $\epsilon^\prime$ is obtained from $\epsilon$ by switching the order of $\epsilon_i$ and $\epsilon_{i+1}$.
Then the directed and labeled edge $\epsilon\overset{i}\rightarrow \epsilon^\prime$ is an edge of $G_{n,k}$.
Hence, an edge of $G_{n,k}$ corresponds to a move of a single 1 in a vertex $\epsilon$ one spot to the right, and such a move can be done if and only if the next spot is occupied by a 0.

We are interested in the circuits of minimal length in the graph $G_{n,k}$.
Such a circuit is called a \emph{minimal circuit}.
Suppose that $\epsilon$ is a vertex in a minimal circuit of $G_{n,k}$.
Then the minimal circuit can be thought of as a sequence of edges in $G_{n,k}$ that moves each 1 in $\epsilon$ into the position of the 1 directly to its right (modulo $n$).  
It follows that a minimal circuit in $G_{n,k}$ has length $n$.
An example of a minimal circuit in $G_{9,3}$ is provided in Figure \ref{fig:facet proof example}.
Notice that for a fixed initial vertex of the minimal circuit the labels of the edges form a permutation $\omega=\omega_1\omega_2\cdots\omega_n\in S_n$, the symmetric group on $n$ elements.   
Following the convention of \cite{lam}, we associate a minimal circuit in $G_{n,k}$ with the permutation consisting of the labels of the edges of the circuit for which $\omega_n=n$.
Let $(\omega)$ denote the minimal circuit in $G_{n,k}$ corresponding to the permutation $\omega\in S_n$ with $\omega_n=n$.
Let $\sigma_{(\omega)}$ denote the convex hull in $\R^n$ of the set of vertices of $(\omega)$.   
Notice that $\sigma_{(\omega)}$ will always be an $(n-1)$-simplex.
\begin{thm}
\cite[Lam and Postnikov]{lam}
\label{circuit triangulation theorem}
The collection of simplices $\sigma_{(\omega)}$ given by the minimal circuits in $G_{n,k}$ are the maximal simplices of a triangulation of the hypersimplex $\Delta_{n,k}$.  We call this triangulation the \emph{circuit triangulation}.  
\end{thm}
Denote the circuit triangulation of $\hypernk$ by $\trink$, and let $\max\trink$ denote the set of maximal simplices of $\trink$.
To simplify notation we will write $\omega$ to denote the simplex $\sigma_{(\omega)}\in\max\trink$.
In \cite{braunsol} it is shown that the collection of simplices in $\trink$ that lie completely within $\rstablenk$ form a triangulation of this polytope.  
We let $\trinkr$ denote this triangulation of $\rstablenk$ and let $\max\trinkr$ denote the set of maximal simplices of $\trinkr$.  
In the following, we compute the facet-defining inequalities for $\rstablenk$ using the nesting of triangulations:
\begin{equation*}
\nabla_{n,k}\supset\nabla_{n,k}^2\supset\nabla_{n,k}^3\supset\cdots\supset\nabla_{n,k}^{\left\lfloor\frac{n}{k}\right\rfloor}.
\end{equation*}
The method by which we will do this is outlined in the following remark.

\begin{rem}
\label{facet proof outline}
To compute the facet-defining inequalities of $\rstablenk$ we first consider the geometry of their associated facet-defining $(n-2)$-flats.
Suppose that $\rstablenk$ is $(n-1)$-dimensional.  
Since $\rstablenkminus\supset\rstablenk$ then a facet-defining $(n-2)$-flat of $\rstablenk$ either also defines a facet of $\rstablenkminus$ or it intersects $\relint\rstablenkminus$, the relative interior of $\rstablenkminus$.  
Therefore, to compute the facet-defining $(n-2)$-flats of $\rstablenk$ it suffices to compute the former and latter collections of $(n-2)$-flats independently.  
To identify the former collection we will use an induction argument on $r$.  
To identify the latter collection we work with pairs of adjacent $(n-1)$-simplices in the set $\max\trinkr$.  
Note that two simplices $u,\omega\in\max\trinkr$ are adjacent (i.e. share a common facet) if and only if they differ by a single vertex.  
Therefore, their common vertices span an $(n-2)$-flat which we will denote by $H_{u,\omega}$.  
Thus, we will identify adjacent pairs of simplices $u\in\max\trinkrminus$ and 
\\ $\omega\in\trinkrdifference$ for which $H_{u,\omega}$ is facet-defining.     
\end{rem}


\subsection{Computing facet-defining inequalities via a nesting of triangulations.}
Suppose $1<k<n-1$.  
In order to prove Theorem~\ref{main theorem} in the fashion outlined by Remark~\ref{facet proof outline} we require a sequence of lemmas.  
Notice that $\rstablenk$ is contained in $\hlhalf$ and $\hlrhalf$ for all $\ell\in[n]$.  
So in the following we simply show that $\hl$ and $\hlr$ form the complete set of facet-defining $(n-2)$-flats.

\begin{lem}
\label{hlpreserved}
Let $1\leq r<\left\lfloor\frac{n}{k}\right\rfloor$.  
For all $\ell\in[n]$, $\hl$ is facet-defining for $\rstablenk$.  
\end{lem}

\begin{proof}
First notice that the result clearly holds for $r=1$.  
So we need only show that $n-1$ affinely independent vertices of $\rstablenk$ lie in $\hl$.  
Hence, to prove the claim it suffices to identify a simplex $\omega\in\max\trinkr$ such that $\hl$ supports a facet of $\omega$.  
Since $r\leq\nkfloor-1$ it also suffices to work with $r = \nkfloor-1$.  

Fix $\ell\in[n]$.  
For $r=\nkfloor-1$ we construct a minimal circuit in the graph $G_{n,k}$ that corresponds to a simplex in $\max\trinkr$ for which $\hl$ is facet-supporting.
To this end, consider the characteristic vector of the $k$-subset 
 $\left\{(\ell-1)-(s-1)r : s\in[k]\right\}\subset[n].$
Denote this characteristic vector by $\epsilon^\ell$, and think of its indices modulo $n$.  
Labeling the $1$ in coordinate $(\ell-1)-(s-1)r$ of $\epsilon^\ell$ as $1_s$, we see that $1_s$ and $1_{s+1}$ are separated by $r-1$ zeros for $s\in[k-1]$. 
That is, the coordinate $\epsilon^\ell_i=0$ for every $(\ell-1)-sr<i<(\ell-1)-(s-1)r$ (modulo $n$), and there are precisely $r-1$ such coordinates.
Moreover, since $kr = k\left(\nkfloor-1\right)\leq n$ then there are at least $r-1$ zeros between $1_1$ and $1_k$.
Hence, this vertex is $r$-stable.
From $\epsilon^\ell$ we can now construct an $r$-stable circuit $(\omega^\ell)$ by moving the $1$'s in $\epsilon^\ell$ one coordinate to the right (modulo $n$), one at a time, in the following pattern:
%
%
	\begin{enumerate}[(1)]
		\item Move $1_1$.
		\item Move $1_1$. Then move $1_2$.  Then move $1_{3}$. $\ldots$ 			Then move $1_k$.
		\item Repeat step (2) $r-1$ more times.
		\item Move $1_1$ until it rests in entry $\ell-1$.
	\end{enumerate}
An example of $(\omega^\ell)$ for $n=9, k=3,$ and $\ell=5$ is provided in Figure~\ref{fig:facet proof example}.
\begin{figure}
	\centering
	\includegraphics[width=0.7\textwidth]{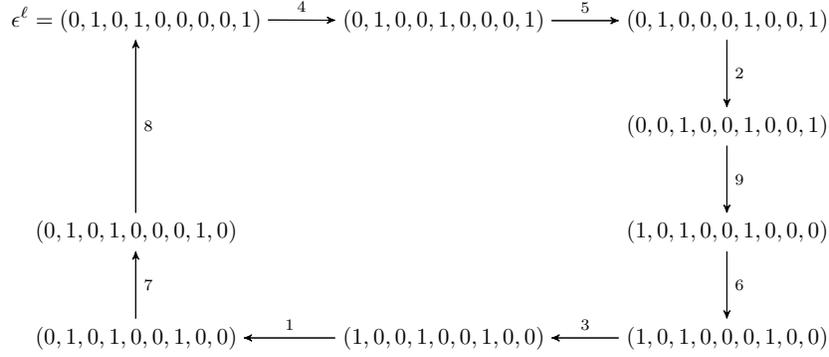}
	\caption{The minimal circuit $(\omega^\ell)$ for $n=9,k=3,$ and $\ell=5$ constructed in Lemma~\ref{hlpreserved}.}
	\label{fig:facet proof example}
\end{figure}
This produces a minimal circuit in $G_{n,k}$ since each $1_s$ has moved precisely enough times to replace $1_{s+1}$.  
Moreover, since $k>1$ then $k\left(\nkfloor-1\right)\leq n-2$.  
So there are at least $r+1$ $0$'s between $1_1$ and $1_k$ in $\epsilon^\ell$.
From here, it is a straight-forward exercise to check that every vertex in $(\omega^\ell)$ is $r$-stable.  
Therefore, $\omega^\ell\in\max\trinkr$.  
Finally, since $r>1$, the simplex $\omega^\ell$ has only one vertex satisfying $x_\ell=1$, and this is the vertex following $\epsilon^\ell$ in the circuit $(\omega^\ell)$.  
Hence, all other vertices of $\omega^\ell$ satisfy $x_\ell=0$.  
So $\hl$ supports a facet of $\omega^\ell$.  
Thus, we conclude that $\hl$ is facet-defining for $\rstablenk$ for $r<\left\lfloor\frac{n}{k}\right\rfloor$.
\end{proof}

The following theorem follows immediately from the construction of the $(n-1)$-simplex $\omega^\ell$ in the proof of Lemma~\ref{hlpreserved}, and it justifies the assumption on the dimension of $\rstablenk$ made in Remark~\ref{facet proof outline}.  
\begin{thm}
\label{dimension theorem}
The polytope $\rstablenk$ is $(n-1)$-dimensional for all $r<\left\lfloor\frac{n}{k}\right\rfloor$.
\end{thm}

\begin{lem}
\label{hlrnotpreserved}
Suppose $r>1$ and $\rstablenk$ is $(n-1)$-dimensional.  
Then $\hlrminus$ is not facet-defining for $\rstablenk$.
\end{lem}

\begin{proof}
Suppose for the sake of contradiction that $\hlrminus$ is facet-defining for $\rstablenk$.  
Since $\rstablenk$ is $(n-1)$-dimensional then there exists an $(n-1)$-simplex $\omega\in\max\trinkr$ such that $\hlrminus$ is facet-defining for $\omega$.
In other words, every vertex in $(\omega)$ satisfies $\sum_{i=\ell}^{\ell+r-2}x_i=1$
except for exactly one vertex, say $\epsilon^\star$.  
Since all vertices in $(\omega)$ are $(0,1)$-vectors, this means all vertices other than $\epsilon^\star$ have exactly coordinate in the subvector $(\epsilon_{\ell},\epsilon_{\ell+1},\ldots,\epsilon_{\ell+r-2})$ being $1$ and all other coordinates are $0$.  
Similarly, this subvector is the $0$-vector for $\epsilon^\star$.  
Since $(\omega)$ is a minimal circuit this means that the move preceding the vertex $\epsilon^\star$ in $(\omega)$ results in the only 1 in $(\epsilon_{\ell},\epsilon_{\ell+1},\ldots,\epsilon_{\ell+r-2})$ exiting the subvector to the right.  
Similarly, the move following the vertex $\epsilon^\star$ in $(\omega)$ results in a single $1$ entering the subvector on the left.  
Suppose that 
	$$\epsilon^\star=(\ldots,\epsilon^\star_{\ell-1},\epsilon^\star_{\ell},\epsilon^\star_{\ell+1},\ldots,\epsilon^\star_{\ell+r-2},\epsilon^\star_{\ell+r-1},\ldots)=(\ldots,1,0,0,\ldots,0,1,\ldots).$$
Then this situation looks like

\begin{center}
\includegraphics[width=0.9\textwidth]{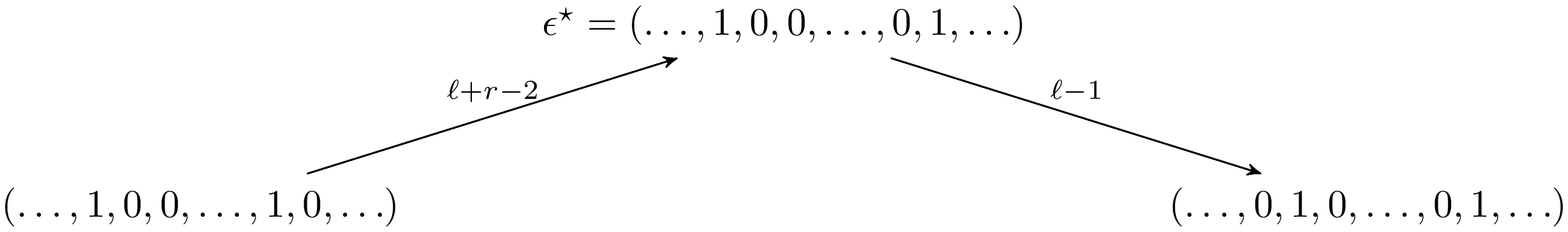}
\end{center}

Hence, neither the vertex preceding or following the vertex $\epsilon^\star$ is $r$-stable.  
For example, in the vertex following $\epsilon^\star$ there is a $1$ in entries $\ell$ and $\ell+r-1$.  
This contradicts the fact that $\omega\in\max\trinkr$.  
\end{proof}

To see why Lemma~\ref{hlrnotpreserved} will be useful suppose that Theorem~\ref{main theorem} holds for $\rstablenkminus$ for some $1\leq r<\left\lfloor\frac{n}{k}\right\rfloor$.
Then Lemmas~\ref{hlpreserved} and \ref{hlrnotpreserved} tell us that the collection of facet-defining $(n-2)$-flats for $\rstablenkminus$ that are also facet-defining for $\rstablenk$ is $\{\hl:\ell\in[n]\}$.  
This is the nature of the induction argument mentioned in Remark~\ref{facet proof outline}.
To identify the facet-defining $(n-2)$-flats of $\rstablenk$ that intersect $\relint\rstablenkminus$ we will use the following definition.  
\begin{defn}
\label{rsupportingpair}
Suppose $u$ and $\omega$ are a pair of simplices in $\max\trink$ satisfying
	\begin{itemize}
		\item $u\in\max\trinkr$, 
		\item $\omega\in\max\trinkrminus\backslash\max\trinkr$, and
		\item $\omega$ uses exactly one vertex that is not $r$-stable, called the \emph{key vertex}, and this is the only vertex by which $u$ and $\omega$ differ.
	\end{itemize}
We say that the ordered pair of simplices $(u,\omega)$ is an \emph{$r$-supporting pair} of $H_{u,\omega}$, where $H_{u,\omega}$ is the flat 
spanned by the common vertices of $u$ and $\omega$.
\end{defn}

\begin{lem}
\label{newhyperplaneshavepairs}
Suppose $1<r<\nkfloor$.  
Suppose also that $H_F$ is a $(n-2)$-flat defining a facet $F$ of $\rstablenk$ such that $H_F\cap\relint\rstablenkminus\neq\emptyset$.  
Then $H_F=H_{u,\omega}$ for some $r$-supporting pair of simplices $(u,\omega)$.
\end{lem}

\begin{proof}
Since $H_F\cap\relint\rstablenkminus\neq\emptyset$ and $\rstablenk$ is contained in $\rstablenkminus$ then \\
$F\cap\relint\rstablenkminus\neq\emptyset$.  
That is, there exists some $\alpha\in F$ such that  $\alpha\in\relint\rstablenkminus$.  
Recall that $\trinkrminus$ is a triangulation of $\rstablenkminus$ that restricts to a triangulation $\trinkr$ of $\rstablenk$.  
It follows that $\trinkr$ and $\trinkrminus\backslash\trinkr$ give identical triangulations of $\partial\rstablenk\cap\relint\rstablenkminus$.    
Since $\rstablenk$ is $(n-1)$-dimensional we may assume, without loss of generality, that $\alpha$ lies in the relative interior of an $(n-2)$-dimensional simplex in the triangulation of $\partial\rstablenk\cap\relint\rstablenkminus$ induced by $\trinkr$ and $\trinkrminus\backslash\trinkr$.
Therefore, there exists some $u\in\max\trinkr$ such that $H_F$ is facet-defining for $u$ and $\alpha\in u\cap H_F$, and there exists some $\omega\in\max\trinkrminus\backslash\max\trinkr$ such that $\alpha\in\omega\cap H_F$.  
Since $\trinkrminus$ is a triangulation of $\rstablenkminus$ it follows that $u\cap H_F=\omega\cap H_F$.
Hence, $u$ and $\omega$ are adjacent simplices that share the facet-defining $(n-2)$-flat $H_F$, and they form an $r$-supporting pair $(u,\omega)$ with $H_{u,\omega}=H_F$.
\end{proof}

It will be helpful to understand the key vertex of an $r$-supporting pair $(u,\omega)$.  
To do so, we will use the following definition.

\begin{defn}
\label{rstablepair}
Let $\epsilon\in\R^n$ be a vertex of $\hypernk$.  
A \emph{pair} of 1's in $\epsilon$ is an ordered pair of two coordinates of $\epsilon$, $(i,j)$, such that $\epsilon_i=\epsilon_j=1$, and $\epsilon_t=0$ for all $i<t<j$ (modulo $n$).  
A pair of 1's is called an \emph{$r$-stable pair} if there are at least $r-1$ 0's separating the two 1's.
\end{defn}

\begin{lem}
\label{hyperplanes for r-supporting pairs}
Suppose $(u,\omega)$ is an $r$-supporting pair, and let $\epsilon$ be the key vertex of this pair.  
Then $\epsilon$ contains precisely one $(r-1)$-stable but not $r$-stable pair, $(\ell,\ell+r-1)$.  
Moreover, $H_{u,\omega}=\hlr$.  
\end{lem}

\begin{proof}
We first show that $\epsilon$ has precisely one $(r-1)$-stable but not $r$-stable pair, $(\ell,\ell+r-1)$.
To see this, consider the minimal circuit $(\omega)$ in the graph $G_{n,k}$ associated with the simplex $\omega$.  Think of the key vertex $\epsilon$ as the initial vertex of this circuit, and recall that each edge of the circuit corresponds to a move of exactly one 1 to the right by exactly one entry.  
Hence, in the circuit $(\omega)$ the vertex following $\epsilon$ differs from $\epsilon$ by a single right move of a single 1.  
Since $\epsilon$ is the only vertex in $(\omega)$ that is $(r-1)$-stable but not $r$-stable then the move of this single 1 to the right by one entry must eliminate all pairs that are $(r-1)$-stable but not $r$-stable.  
Moreover, this move cannot introduce any new $(r-1)$-stable but not $r$-stable pairs.  
Since a single 1 can be in at most two pairs, and this 1 must move exactly one entry to the right, then this 1 must be in entry $j$ in the pairs $(i,j)$ and $(j,t)$ where $(i,j)$ is $(r-1)$-stable but not $r$-stable, and $(j,t)$ is $(r+1)$-stable.  
Moreover, since the move of the 1 in entry $j$ can only change the stability of the pairs $(i,j)$ and $(j,t)$ then it must be that all other pairs are $r$-stable.

Finally, since $\omega$ has the unique $(r-1)$-stable but not $r$-stable vertex $\epsilon$, and since $\epsilon$ has the unique $(r-1)$-stable but not $r$-stable pair $(\ell,\ell+r-1)$ then all other vertices in $\omega$ satisfy $\sum_{i=\ell}^{\ell+r-1}x_i=1$.  
Hence, $H_{u,\omega}=\hlr$.  
\end{proof}

\begin{lem}\label{thenewhyperplanes}
Suppose $1<r<\nkfloor$.  
Suppose also that $H_F$ is an $(n-2)$-flat defining a facet $F$ of $\rstablenk$ and $H_F\cap\relint\rstablenkminus\neq\emptyset$.  
Then $H_F=\hlr$ for some $\ell\in[n]$.  
\end{lem}

\begin{proof}
By Lemma~\ref{newhyperplaneshavepairs} the $H_F=H_{u,\omega}$ for some $r$-supporting pair $(u,\omega)$.  
By Lemma~\ref{hyperplanes for r-supporting pairs} $\omega$ has a unique vertex that is $(r-1)$-stable but not $r$-stable with a unique $(r-1)$-stable but not $r$-stable pair $(\ell, \ell+r-1)$ for some $\ell\in[n]$.  
Thus, $H_F=H_{u,\omega}=\hlr$.
\end{proof}

We now show that $\hlr$ is indeed facet-defining for $\rstablenk$ for all $\ell\in[n]$.

\begin{lem}\label{allthenewhyperplanes}
Suppose $1\leq r<\left\lfloor\frac{n}{k}\right\rfloor$ or $n=kr+1$.  
Then $\hlr$ is facet-defining for $\rstablenk$ for all $\ell\in[n]$.
\end{lem}

\begin{proof}
First we note that the result is clearly true for $r=1$.
So in the following we assume $r>1$.
To prove the claim we show that $\hlr$ supports an $(n-1)$-simplex $\omega\in\max\trinkr$.

To this end, consider the characteristic vector of the $k$-subset 
\\ $\left\{(\ell-1)+(s-1)r : s\in[k]\right\}\subset[n].$
Denote this characteristic vector by $\epsilon^\ell$, and think of its indices modulo $n$.  
Labeling the $1$ in coordinate $(\ell-1)+(s-1)r$ of $\epsilon^\ell$ as $1_s$, it is quick to see that $1_s$ and $1_{s+1}$ are separated by $r-1$ zeros for every $s\in[k]$. 
That is, $\epsilon^\ell_i=0$ for every $(\ell-1)+(s-1)r<i<(\ell-1)+sr$ (modulo $n$), and there are precisely $r-1$ such coordinates.
Moreover, since $r<\nkfloor$ or $n=kr+1$ then $n\geq kr+1$.  
So there are at least $r$ zeros between $1_1$ and $1_k$.
Hence, this vertex is $r$-stable.
From $\epsilon^\ell$ we can now construct an $r$-stable circuit $(\omega^\ell)$ by moving the $1$'s in $\epsilon^\ell$ one coordinate to the right (modulo $n$), one at a time, in the following pattern:
	\begin{enumerate}[(1)]
		\item Move $1_k$. Then move $1_{k-1}$.  Then move $1_{k-2}$. $\ldots$ Then move $1_1$.  
		\item Repeat step (1) $r-1$ more times.
		\item Move $1_k$ to entry $\ell$.
	\end{enumerate}
Each move in this pattern produces a new $r$-stable vertex since there are always at least $r-1$ zeros between each pair of $1$'s.
So $\omega^\ell\in\max\trinkr$ and $\hlr$ supports $\omega^\ell$ since every vertex of $(\omega^\ell)$ lies in $\hlr$ except for the vertex preceding the first move of $1_1$ in the circuit $(\omega^\ell)$.  
\end{proof}

\begin{rem}
When $n=kr+1$ then $\omega^\ell =\rstablenk$ for all $\ell\in[n]$.
So the facet-defining inequalities for $\omega^\ell = \rstablenk$ are precisely $\hlrhalf$ for $\ell\in[n]$.  
\end{rem}
From Lemmas~\ref{thenewhyperplanes} and \ref{allthenewhyperplanes} we see that when $1<r<\nkfloor$ the facet-defining $(n-2)$-flats for $\rstablenk$ that intersect $\relint\rstablenkminus$ are precisely $\hlr$ for $\ell\in[n]$.  
We are now ready to prove Theorem~\ref{main theorem}.


\subsubsection{Proof of Theorem~\ref{main theorem}}
First recall that Theorem~\ref{main theorem} is known to be true for $r=1$.  
Now let $1<r<\left\lfloor\frac{n}{k}\right\rfloor$.  
By Theorem~\ref{dimension theorem} we know that $\rstablenk$ is $(n-1)$-dimensional.  
First let $r=2$.  
By Lemma~\ref{hlpreserved} we know that $\hl$ is facet-defining for $\twostablenk$ for all $\ell\in[n]$.  
By Lemma~\ref{hlrnotpreserved} we know that for every $\ell\in[n]$ $H_{\ell,1}$ is not facet-defining for $\twostablenk$.  
Thus, the collection of facet-defining $(n-2)$-flats for $\hypernk$ that are also facet-defining for $\twostablenk$ are $\{\hl:\ell\in[n]\}$, and all other facet-defining $(n-2)$-flats for $\twostablenk$ must intersect the relative interior of $\hypernk$.  
Therefore, by Lemmas~\ref{thenewhyperplanes} and \ref{allthenewhyperplanes} the remaining facet-defining $(n-2)$-flats for $\twostablenk$ are $H_{\ell,2}$ for $\ell\in[n]$.
Since $\rstablenk$ is contained in $\hlhalf$ and $\hlrhalf$, this proves the result for $r=2$.  
Theorem~\ref{main theorem} then follows by iterating this argument for $2<r<\left\lfloor\frac{n}{k}\right\rfloor$.


\section{Gorenstein $r$-stable Hypersimplices}
\label{gorenstein}

In \cite{braunsol}, the authors note that the $r$-stable hypersimplices appear to have unimodal Ehrhart $\delta$-vectors, and they verify this observation for a collection of these polytopes in the $k=2$ case.  
In \cite{bruns}, it is shown that a Gorenstein integer polytope with a regular unimodular triangulation has a unimodal $\delta$-vector.  
In \cite{braunsol}, it is shown that $\rstablenk$ has a regular unimodular triangulation.  
One application for the equations of the facets of a rational convex polytope is to determine whether or not the polytope is Gorenstein \cite{hibi6}.  
We now utilize Theorem~\ref{main theorem} to identify $1\leq r<\nkfloor$ for which $\rstablenk$ is Gorenstein.  
We identify a collection of such polytopes for every $k\geq2$, thereby expanding the collection of $r$-stable hypersimplices known to have unimodal $\delta$-vectors.
In this section we let $1<k<n-1$.
This is because $\Delta_{n,1}$ and $\Delta_{n,n-1}$ are simply copies of the standard $(n-1)$-simplex, which are well-known to be Gorenstein \cite[p.29]{beck}.

First we recall the definition of a Gorenstein polytope.  
Let $P\subset\R^N$ be a rational convex polytope of dimension $d$, and for an integer $q\geq1$ let $qP:=\{q\alpha:\alpha\in P\}$.  
Let $x_1,x_2,\ldots, x_N,$ and $z$ be indeterminates  over some field $K$.  
Given an integer $q\geq1$, let $A(P)_q$ denote the vector space over $K$ spanned by the monomials $x_1^{\alpha_1}x_2^{\alpha_2}\cdots x_N^{\alpha_N}z^q$ for $(\alpha_1,\alpha_2,\ldots,\alpha_N)\in qP\cap\ZZ^N$.  
Since $P$ is convex we have that $A(P)_pA(P)_q\subset A(P)_{p+q}$ for all $p$ and $q$.  
It then follows that the graded algebra
	$$A(P):=\bigoplus_{q=0}^\infty A(P)_q$$
is finitely generated over $K=A(P)_0$.  
We call $A(P)$ the \emph{Ehrhart Ring} of $P$, and we say that $P$ is \emph{Gorenstein} if $A(P)$ is Gorenstein.

We now recall the combinatorial criterion given in \cite{den} for an integral convex polytope $P$ to be Gorenstein.  
Let $\partial P$ denote the boundary of $P$ and let $\relint(P)=P-\partial P$.  
We say that $P$ is of \emph{standard type} if $d=N$ and the origin in $\R^d$ is contained in $\relint(P)$.  
When $P\subset \R^d$ is of standard type we define its polar set
	$$P^\star=\left\{(\alpha_1,\alpha_2,\ldots,\alpha_d)\in\R^d:\sum_{i=1}^d\alpha_i\beta_i\leq1\mbox{ for every $(\beta_1,\beta_2,\ldots,\beta_d)\in P$}\right\}.$$
The polar set $P^\star$ is again a convex polytope of standard type, and $(P^\star)^\star=P$.  We call $P^\star$ the \emph{dual polytope} of $P$.  
Suppose $(\alpha_1,\alpha_2,\ldots,\alpha_d)\in\R^d$, and $K$ is the hyperplane in $\R^d$ defined by the equation $\sum_{i=1}^d\alpha_ix_i=1$.
A well-known fact is that $(\alpha_1,\alpha_2,\ldots,\alpha_d)$ is a vertex of $P^\star$ if and only if $K\cap P$ is a facet of $P$.  
It follows that the dual polytope of a rational polytope is always rational.  
However, it need not be that the dual of an integral polytope is always integral.  
If $P$ is an integral polytope with integral dual we say that $P$ is \emph{reflexive}.  
This idea plays a key role in the following combinatorial characterization of Gorenstein polytopes.

\begin{thm}\cite[De Negri and Hibi]{den}\label{reflexivetheorem}
Let $P\subset\R^d$ be an integral polytope of dimension $d$, and let $q$ denote the smallest positive integer for which 
	$$q(\relint(P))\cap\ZZ^d\neq\emptyset.$$
Fix an integer point $\alpha\in q(\relint(P))\cap\ZZ^d$, and let $Q$ denote the integral polytope $q P-\alpha\subset\R^d$.  
Then the polytope $P$ is Gorenstein if and only if the polytope $Q$ is reflexive.
\end{thm}

Since Theorem \ref{reflexivetheorem} requires that the polytope be full-dimensional we consider $\varphi^{-1}\left(\rstablenk\right)$, 
where $\varphi:\R^{n-1}\longrightarrow H$ is the affine isomorphism 
$$\varphi:(\alpha_1,\alpha_2,\ldots,\alpha_{n-1})\longmapsto\left(\alpha_1,\alpha_2,\ldots,\alpha_{n-1},k-\left(\sum_{i=1}^{n-1}\alpha_i\right)\right).$$  
Notice that $\varphi$ is also a lattice isomorphism.  
Hence, we have the isomorphism of Ehrhart Rings as graded algebras
	$$A\left(\varphi^{-1}\left(\rstablenk\right)\right)\cong A\left(\rstablenk\right).$$
Let $\pnk:=\varphi^{-1}\left(\rstablenk\right)$, and recall from Theorem \ref{main theorem} that
   $$\rstablenk=\left(\bigcap_{\ell=1}^n\hlhalf\right)\cap\left(\bigcap_{\ell=1}^n\hlrhalf\right).$$


\subsection{The $H$-representation for $\pnk$.}

We now give a description of the facet-defining inequalities for $\pnk$ in terms of those defining $\rstablenk$.
In the following, it will be convenient to let $T=\{\ell,\ell+1,\ell+2,\ldots,\ell+r-1\}$ for $\ell\in[n]$.
We also let $T^c$ denote the complement of $T$ in $[n]$.  
Notice that for a fixed $1\leq r<\nkfloor$ and $\ell\in[n]$, the set $T$ is precisely the set of summands in the defining equation of the $(n-2)$-flat $\hlr$. 
The defining inequalities of $\pnk$ corresponding to the $(n-2)$-flats $\hlr$ come in two types, dependent on whether $n\notin T$ or $n\in T$.  
If $n\notin T$ then
$$\phlrhalf:=\varphi^{-1}\left(\hlrhalf\right)=\left\{(x_1,x_2,\ldots,x_{n-1})\in\R^{n-1}:\sum_{i\in T}x_i\leq1\right\}.$$
If $n\in T$ then
$$\phlrhalfn:=\varphi^{-1}\left(\hlrhalf\right)=\left\{(x_1,x_2,\ldots,x_{n-1})\in\R^{n-1}:\sum_{i\in T^c}x_i\geq k-1\right\}.$$
Similarly, if $\ell\neq n$ then 
$$\phlhalf:=\varphi^{-1}\left(\hlhalf\right)=\left\{(x_1,x_2,\ldots,x_{n-1})\in\R^{n-1}:x_i\geq0\right\}.$$
Finally, if $\ell=n$ then
$$\phlnhalf:=\varphi^{-1}\left(H_{n}^{(+)}\right)=\left\{(x_1,x_2,\ldots,x_{n-1})\in\R^{n-1}:\sum_{i=1}^{n-1}x_i\leq k \right\}.$$
Thus, we may write $\pnk$ as the intersection of closed halfspaces in $\R^{n-1}$
	$$\pnk=\left(\bigcap_{n\notin T}\phlrhalf\right)\cap\left(\bigcap_{n\in T}\phlrhalfn\right)\cap\left(\bigcap_{i=1}^{n-1}\phlhalf\right)\cap\phlnhalf.$$
To denote the supporting hyperplanes corresponding to these halfspaces we simply drop the superscripts $(+)$ and $(-)$.


\subsection{The codegree of $\pnk$.}

Given the above description of $\pnk$, we would now like to determine the smallest positive integer $q$ for which $q\pnk$ contains a lattice point in its relative interior.  
To do so, recall that for a lattice polytope $P$ of dimension $d$ we can define the \emph{(Ehrhart) $\delta$-polynomial} of $P$.  
If we write this polynomial as
	$$\delta_P(z)=\delta_0+\delta_1z+\delta_2z^2+\cdots+\delta_dz^d$$
then we call the coefficient vector $\delta(P)=(\delta_0,\delta_1,\delta_2,\ldots,\delta_d)$ the \emph{$\delta$-vector} of $P$.  
We let $s$ denote the degree of $\delta_P(z)$, and we call $q=(d+1)-s$ the \emph{codegree} of $P$.  
It is a consequence of Ehrhart Reciprocity that $q$ is the smallest positive integer such that $q P$ contains a lattice point in its relative interior \cite{beck}.  
Hence, we would like to compute the codegree of $\pnk$.  
To do so requires that we first prove two lemmas.
In the following let $q=\nkceiling$.
Our first goal is to show that there is at least one integer point in $\relint\left(q\pnk\right)$ for $1\leq r<\nkfloor$.
We then show that $q$ is the smallest positive integer for which this is true.
Recall that $q=\frac{n+\alpha}{k}$ for some $\alpha\in\{0,1,\ldots,k-1\}$.  
Also recall that for a fixed $n$ and $k$ we have the nesting of polytopes
\begin{equation*}
P_{n,k}\supset P_{n,k}^{\stab(2)}\supset P_{n,k}^{\stab(3)}\supset\cdots\supset P_{n,k}^{\stab\left(\left\lfloor\frac{n}{k}\right\rfloor-1\right)}\supset P_{n,k}^{\stab\left(\left\lfloor\frac{n}{k}\right\rfloor\right)}.
\end{equation*}
Hence, if we identify an integer point inside $\relint\left(q\floorminusstablenk\right)$ then this same integer point lives inside $\relint\left(q\pnk\right)$ for every $1\leq r<\nkfloor$.
Given these facts, we now prove two lemmas.

\begin{lem}\label{alphaoneorzero}
Suppose that $q=\nkceiling=\frac{n+\alpha}{k}$ where $\alpha\in\{0,1\}$.  Then the integer point $(1,1,\ldots,1)\in\R^{n-1}$ lies inside $\relint\left(q\pnk\right)$ for every $1\leq r<\nkfloor$.
\end{lem}

\begin{proof}
It suffices to show that $(x_1,x_2,\ldots,x_{n-1})=(1,1,\ldots,1)$ satisfies the set of inequalities

	\begin{enumerate}[(i)]
	\item $x_i>0$,  for $i\in[n-1]$,
	\item $\sum_{i=1}^{n-1}x_i< kq$,
	\item $\sum_{i\in T}x_i<q$, for $n\notin T$, and 
	\item $\sum_{i\in T^c}x_i>(k-1)q$, for $n\in T$.
	\end{enumerate}
We do this in two cases.  
First suppose that $\alpha=0$.  
Then $k$ divides $n$ and $q=\frac{n}{k}$.  
Clearly, (i) is satisfied.  
To see that (ii) is also satisfied simply notice that $n-1<kq$.
To see that (iii) is satisfied recall that $\#T=r$ and $r<\nkfloor=q$.  
Finally, to see that (iv) is satisfied notice that $\#T^c=n-r$.  
So we would like that $n-r>(k-1)q$.  
However, this follows quickly from the fact that $r<\frac{n}{k}$.

Now consider the case where $\alpha=1$.  
Recall that it suffices to consider the case when $r=\nkfloor-1$.
Inequalities (i), (ii), and (iii) are all satisfied in the same fashion as the case when $\alpha=0$.  
So we need only check that (iv) is also satisfied.  
Again we would like that $n-r>(k-1)q$.  
Notice since $\alpha=1$ then $k$ does not divide $n$, and so $\nkceiling=\nkfloor+1$.
Hence, $q=r+2$.
The desired inequality then follows from $n+2>n+\alpha$. 
Thus, whenever $\alpha\in\{0,1\}$, the lattice point $(1,1,\ldots,1)\in\relint\left(q\pnk\right)$ for every $1\leq r<\nkfloor$.
\end{proof}

Next we would like to identify an integer point in the relative interior of $q\pnk$ for $1\leq r<\nkfloor$ when $\alpha\geq2$.  
In this case, the point $(1,1,\ldots,1)$ does not always work, so we must identify another point.  
Recall that it suffices to identify such a point for $r=\nkfloor-1$.
To do so, we construct the desired point using the notions of $r$-stability.  
Fix $n$ and $k$ such that $q=\frac{n+\alpha}{k}$ for $\alpha\geq2$, and let $r=\nkfloor-1$.
This also fixes the value $\alpha\in\{2,3,\ldots,k-1\}$.
Since $r=\nkfloor-1$ we may construct an $r$-stable vertex in $\R^n$ as the characteristic vector of the set 
	$$\{n-r,n-2r,n-3r,\ldots,n-(k-1)r\}\subset[n].$$
Notice that there are at least $r$ 0's between the $n^{th}$ coordinate of the vertex and the $n-(k-1)r^{th}$ coordinate (read from right-to-left modulo $n$).
In particular, this implies that the $n^{th}$ coordinate (and the $1^{st}$ coordinate) is occupied by a 0.
To construct the desired vertex replace the 1's in coordinates 
	$$n-(\alpha+1)r,n-(\alpha+2)r,\ldots,n-(k-1)r$$
with 0's.
Now add 1 to each coordinate of this lattice point.
If the resulting point is $(x_1,x_2,\ldots,x_n)$ then replace $x_n=1$ with the value $kq-\left(\sum_{i=1}^{n-1}x_i\right).$
Call the resulting vertex $\epsilon^\alpha$, and consider the isomorphism $\widetilde{\varphi}:\R^{n-1}\longrightarrow H_q$, defined analogously to $\varphi$, where
$H_q$ is the hyperplane in $\R^n$ defined by the equation $\sum_{i=1}^nx_i=kq$.
Notice that by our construction of $\epsilon^\alpha$, the point $\widetilde{\varphi}^{-1}\left(\epsilon^\alpha\right)$ is simply $\epsilon^\alpha$ with the last coordinate projected off.

\begin{lem}\label{alphageqtwo}
Suppose that $q=\nkceiling=\frac{n+\alpha}{k}$ for $\alpha\in\{2,3,\ldots,k-1\}$.  
Then the lattice point $\widetilde{\varphi}^{-1}\left(\epsilon^\alpha\right)$ lies inside $\relint\left(q\pnk\right)$ for every $1\leq r<\nkfloor$. 

\end{lem}

\begin{proof}
It suffices to show that when $r=\nkfloor-1$ the lattice point  \\
$(x_1,x_2,\ldots,x_{n-1})=\widetilde{\varphi}^{-1}\left(\epsilon^\alpha\right)$ satisfies inequalities (i), (ii), (iii), and (iv) from the proof of Lemma \ref{alphaoneorzero}.
It is clear that (i) is satisfied.  
To see that (ii) is satisfied notice that $\sum_{i=1}^{n-1}x_i=n-1+\alpha.$
This is because $\alpha$ coordinates of $\widetilde{\varphi}^{-1}$ are occupied by $2$'s and all other coordinates are occupied by $1$'s.  
Thus, inequality (ii) is satisfied since $n-1+\alpha<kq$.
To see that (iii) is satisfied first notice that for $T$ with $n\notin T$
	$$\sum_{i\in T}x_i=\begin{cases} r & \mbox{ if $T$ contains no entry with value $2$,} \\ r+1 & \mbox		{ otherwise.}\end{cases}$$
This is because we have chosen the 2's to be separated by at least $r-1$ 0's.
Thus, since $k$ does not divide $n$ we have that $\sum_{i\in T}x_i\leq r+1=\nkfloor<q.$
Finally, to see that (iv) is satisfied first notice that for $T$ with $n\in T$
	$$\sum_{i\in T^c}x_i=\begin{cases} n-r+\alpha-1 & \mbox{ if $T$ contains an entry with value $2$,} 		\\ n-r+\alpha & \mbox{ otherwise.}\end{cases}$$
Hence, we must show that $n-r+\alpha-1>(k-1)q$.  
However, since $\nkceiling=\nkfloor+1$ then $r=q-2$, and so the desired inequality follows from $n+\alpha+1>n+\alpha$.
Therefore, $\widetilde{\varphi}^{-1}\left(\epsilon^\alpha\right)\in\relint\left(q\pnk\right)$ for every $1\leq r<\nkfloor$.
\end{proof}

Using Lemmas \ref{alphaoneorzero} and \ref{alphageqtwo} we now show that $q=\nkceiling$ is indeed the codegree of these polytopes.

\begin{thm}\label{codegree}
Let $1\leq r<\nkfloor$.  The codegree of $\pnk$ is $q=\nkceiling$.
\end{thm}

\begin{proof}
First recall that $\pnk$ is a subpolytope of $\hypernk$.  
By a theorem of Stanley \cite{stan4} it then follows that $\delta\left(\pnk\right)\leq\delta\left(\hypernk\right)$.  
Therefore, the codegree of $\pnk$ is no smaller than the codegree of $\hypernk$.  
In \cite[Corollary 2.6]{kat}, Katzman determines that the codegree of $\hypernk$ is $q=\nkceiling$.  
Since Lemmas \ref{alphaoneorzero} and \ref{alphageqtwo} imply that $q\pnk$ contains a lattice point inside its relative interior we conclude that the codegree of $\pnk$ is $q=\nkceiling$.
\end{proof}

Recall that if an integral polytope $P$ of dimension $d$ with codegree $q$ is Gorenstein then 
	$$\#\left(\relint\left(qP\right)\cap\ZZ^d\right)=1.$$
With this fact in hand, we have the following corollary.

\begin{cor}\label{alphageqtwonotgorenstein}
Suppose that $q=\nkceiling=\frac{n+\alpha}{k}$, where $\alpha\in\{2,3,\ldots,k-1\}$.  
Then the polytope $\rstablenk$ is not Gorenstein for every $1\leq r<\nkfloor$.
\end{cor}

\begin{proof}
Recall the vertex $(x_1,x_2,\ldots,x_n)$ from which we produce $\epsilon^\alpha$.
Since $x_1=1$ then cyclically shifting the entries of this vertex one entry to the left, and then applying the construction for $\epsilon^\alpha$ results in a second vertex, say $\zeta^\alpha$, such that $\widetilde{\varphi}\left(\zeta^\alpha\right)^{-1}$ also lies in the relative interior of $q\pnk$. 
Thus, $\#\left(\relint\left(q\pnk\right)\cap\ZZ^d\right)>1,$ and we conclude that $\rstablenk$ is not Gorenstein.
\end{proof}


\subsection{Gorenstein $r$-stable hypersimplices and unimodal $\delta$-vectors.}

Notice that by Corollary \ref{alphageqtwonotgorenstein} we need only consider those $r$-stable hypersimplices satisfying the conditions of Lemma \ref{alphaoneorzero}.
For these polytopes we now consider the translated integral polytope
	$$Q:=q\pnk-(1,1,\ldots,1).$$
From our $H$-representation of $\pnk$ we see that the facets of $Q$ are supported by the hyperplanes
	\begin{enumerate}[(a)]
	\item $x_i=-1$,  for $i\in[n-1]$,
	\item $\sum_{i=1}^{n-1}x_i=kq-(n-1)$,
	\item $\sum_{i\in T}x_i=q-r$, for $n\notin T$, and 
	\item $\sum_{i\in T^c}x_i=(k-1)q-(n-r)$, for $n\in T$.
	\end{enumerate}
Given this collection of hyperplanes we may now prove the following theorem.

\begin{thm}\label{gorensteinrstablethm}
Let $1\leq r<\nkfloor$.  Then $\rstablenk$ is Gorenstein if and only if $n=kr+k$.
\end{thm}

\begin{proof}
By Theorem \ref{reflexivetheorem} we must determine when all the vertices of $Q^\star$ are integral.  
We do so by means of the inclusion-reversing bijection between the faces of $Q$ and the faces of $Q^\star$.  
It is immediate that the vertices of $Q^\star$ corresponding to hyperplanes given in (a) are integral.  
So consider the hyperplane given in (b).  
Recall that $q=\nkceiling=\frac{n+\alpha}{k}$ for some $\alpha\in\{0,1\}$. 
Hence, this hyperplane is equivalently expressed as
	$$\sum_{i=1}^{n-1}\frac{1}{\alpha+1}x_i=1.$$
Therefore, the corresponding vertex in $Q^\star$ is integral only if $\alpha=0$.  
Notice next that the hyperplanes given in (c) will have corresponding vertex of $Q^\star$ integral only if $q-r=1$.  
Since $\alpha=0$ we have that $q=\frac{n}{k}$ where $k$ divides $n$, and so it must be that $n=kr+k$.
Finally, when $n=kr+k$ the hyperplanes given in (d) reduce to 
	$$\sum_{i\in T^c}x_i=-1.$$
Hence, the corresponding vertex of $Q^\star$ is integral, and we conclude that, for $1\leq r<\nkfloor$, the polytope $\rstablenk$ is Gorenstein if and only if $n=kr+k$.
\end{proof}

Theorem \ref{gorensteinrstablethm} demonstrates that the Gorenstein property is quite rare amongst the $r$-stable hypersimplices.  
It also enables us to expand the collection of $r$-stable hypersimplices known to have unimodal $\delta$-vectors.
Previously, this collection was limited to the case when $k=2$ or when $\rstablenk$ is a simplex \cite{braunsol}.
Theorem \ref{gorensteinrstablethm} provides a collection of $r$-stable hypersimplices with unimodal $\delta$-vectors for every $k\geq1$.

\begin{cor}\label{unimodalgorenstein}
Let $k\geq1$.  The $r$-stable $n,k$-hypersimplices $\rstablenk$ for $r\geq1$ and $n=kr+k$ have unimodal $\delta$-vectors.
\end{cor}

\begin{proof}
By \cite[Corollary 2.6]{braunsol} there exists a regular unimodular triangulation of $\rstablenk$.  
By Theorem \ref{gorensteinrstablethm} the polytope $\rstablenk$ is Gorenstein for $n=kr+k$ when $k>1$.  
By \cite[Theorem 1]{bruns} we conclude that the $\delta$-vector of $\rstablenk$ is unimodal.
Finally, notice that when $k=1$ these polytopes are just the standard $(n-1)$-simplices.
\end{proof}



\begin{thebibliography}{20}



\bibitem{beck} M. Beck and S. Robins. \emph{Computing the continuous discretely: Integer-point enumeration in polyhedra}. Springer, 2007.


\bibitem{braunsol} B. Braun and L. Solus.  \emph{A shelling of the odd second hypersimplex}. Preprint, 2014.

\bibitem{bruns} W. Bruns and T. R\"{o}mer. \emph{$h$-Vectors of Gorenstein polytopes}. Journal of Combinatorial Theory, Series A 114.1 (2007): 65-76.

\bibitem{del} J. De Loera, D. Haws, and M. K\"{o}ppe.  \emph{Ehrhart polynomials of matroid polytopes and polymatroids}.  Discrete and Computational Geometry., Vol 42, Issue 4, pp. 670-702, 2009.


\bibitem{den} E. De Negri and T. Hibi. \emph{Gorenstein algebras of Veronese type}.  J. Algebra, Vol 193, Issue 2, 629-639 (1997).




\bibitem{hibi6} T. Hibi. \emph{Dual polytopes of rational convex polytopes}. Combinatorica 12.2 (1992): 237-240.





\bibitem{kat} M. Katzman.  \emph{The Hilbert series of algebras of Veronese type}.  Communications in Algebra, Vol 33, pp. 1141-1146, 2005.

\bibitem{lam} T. Lam and A. Postnikov.  \emph{Alcoved Polytopes, I}. Discrete and Computational Geometry., Vol 38, Issue 3, pp. 453-478, 2007.



\bibitem{stan4} R. Stanley. \emph{A Monotonicity Property of $h$-vectors and $h^*$-vectors}. European Journal of Combinatorics 14.3 (1993): 251-258.







\end{thebibliography}
\end{document}